\documentclass[12pt]{amsart}

\usepackage[left=1in, right=1in, top=1in, bottom=1in]
{geometry}
\usepackage{amsfonts,amssymb,graphicx,amsmath,amsthm, float}
\usepackage{xypic}
\usepackage{epstopdf}

\theoremstyle{plain}
\newtheorem{theorem}{Theorem}[section]

\newtheorem{example}[theorem]{Example}

\newtheorem{corollary}[theorem]{Corollary}

\newcommand{\Q}{\mathbb{Q}}

\newcommand{\N}{\mathbb{N}}

\theoremstyle{remark}
\newtheorem*{remark}{Remark}

\begin{document}

\title{Generalized Irreducible Divisor Graphs}
        \date{\today}

\author{Christopher Park Mooney}
\address{Reinhart Center \\ Viterbo University \\ 900 Viterbo Drive \\ La Crosse, WI 54601}
\email{cpmooney@viterbo.edu}


\begin{abstract}
In 1988, I. Beck introduced the notion of a zero-divisor graph of a commutative rings with $1$.  There have been several generalizations in recent years.  In particular, in 2007 J. Coykendall and J. Maney developed the irreducible divisor graph.  Much work has been done on generalized factorization, especially $\tau$-factorization.  The goal of this paper is to synthesize the notions of $\tau$-factorization and irreducible divisor graphs in domains.  We will define a $\tau$-irreducible divisor graph for non-zero non-unit elements of a domain.  We show that by studying $\tau$-irreducible divisor graphs, we find equivalent characterizations of several finite $\tau$-factorization properties.
\\
\vspace{.1in}\noindent \textbf{2010 AMS Subject Classification:} 13A05, 13E99, 13F15, 5C25
\end{abstract}
\maketitle
\section{Introduction}  
\indent In this article, $D$ will denote an integral domain.  We will always assume that all rings have an identity which is not zero.  We will use $G=(V,E)$ to denote a graph with $V$ the set of vertices and $E$ the set of edges.  In 1988, I. Beck in \cite{Beck}, introduced for a commutative ring $R$, the notion of a zero-divisor graph $\Gamma(R)$.  The vertices of $\Gamma(R)$ are the set of zero-divisors, and there is an edge between $a,b \in Z(R)$ if $ab=0$.  This has been studied and developed by many authors including, but not limited to D.D. Anderson, D.F. Anderson, A. Frazier, A. Lauve, P.S. Livingston, and M. Naseer in \cite{andersonzdg, davidanderson, Livingston}.  
\\
\indent There have been several generalizations and extensions of this concept.  In this paper, we focus on the notion of an irreducible divisor graph first formulated by J. Coykendall and J. Maney in \cite{Coykendall}.  Instead of looking exclusively at divisors of zero in a ring, they restrict to a domain $D$ and choose any non-zero, non-unit $x\in D$ and study the relationships between the irreducible divisors of $x$.  In \cite{Axtellidgzd, Coykendall}, M. Axtell, N. Baeth, J. Coykendall, J. Maney, and J. Stickles present several nice results about factorization properties of domains based on their associated irreducible divisor graphs.  M. Axtell and J. Stickles have also studied irreducible divisor graphs in commutative rings with zero-divisors in \cite{Axtellidgd}.
\\
\indent Unique factorization domains are well known and widely studied.  Lesser known are other finite factorization properties that domains might possess which are weaker than UFDs.  These come in the form of half factorization domains or half factorial domains (HFDs), finite factorization domains (FFDs) and bounded factorization domains (BFDs).  See \cite{anderson90} for more information on the developments in the theory of factorization in integral domains.  More recently, these concepts have been further generalized by way of $\tau$-factorization in several papers, especially by D.D. Anderson and A. Frazier in \cite{Frazier} as well as the author in \cite{Mooney, Mooney2}.  In this paper, we seek to take the notion of $\tau$-factorization and apply it to irreducible divisor graphs.  We will find that many equivalent characterizations of $\tau$-finite factorization properties given in the aforementioned papers can be given by studying $\tau$-irreducible divisor graphs.
\\
\indent Section Two gives the necessary preliminary background information and definitions from the study of irreducible and zero-divisor graphs as well as $\tau$-factorization.  In Section Three, we define the $\tau$-irreducible divisor graph of a domain $D$ with a fixed relation $\tau$.  We give a few examples of $\tau$-irreducible divisor graphs, especially in comparison with the irreducible divisor graphs of \cite{Axtellidgd, Coykendall}.  In Section Four, we prove several theorems illustrating how $\tau$-irreducible divisor graphs give us another way to characterize various $\tau$-finite factorization properties domains may possess as defined in \cite{Frazier}.  

\section{Preliminary Definitions}
\subsection{Irreducible Divisor Graph Definitions}\ 
\\
\indent We begin with some definitions from J. Coykendall and J. Maney \cite{Coykendall}.  Let $Irr(D)$ be the set of all irreducible elements in a domain $D$.  We will let $\overline{Irr}(D)$ be a (pre-chosen) set of coset representatives of the collection $\{a U(D) \mid a \in Irr(D)\}$.  Let $x\in D^{\#}$ have a factorization into irreducibles.  The irreducible divisor graph of $x \in D^{\#}$, will be the graph $G(x)=(V,E)$ where $V=\{a\in \overline{Irr}(D) \mid a \mid x\}$, i.e. the set of irreducible divisors of $x$ up to associate.  Given $a_1, a_2 \in \overline{Irr}(D)$, $a_1a_2 \in E$ if and only if $a_1a_2 \mid x$.  Furthermore, $n-1$ loops will be attached to $a$ if $a^n \mid x$.  If arbitrarily large powers of $a$ divide $x$, we allow an infinite number of loops.  They define the \emph{reduced irreducible divisor graph} of $x$ to be the subgraph of $G(x)$ which is formed by deleting all the loops and denote it as $\overline{G}(x)$.
\\

\indent A \emph{clique} will refer to a simple (no loops or multiple edges), complete (all vertices are pairwise adjacent) graph.  A clique on $n \in \N$ vertices will be denoted $K_n$.  We will call a graph $G$ a \emph{pseudoclique} if $G$ is a complete graph having some number of loops (possibly zero).  This means a clique is still considered a pseudoclique.
\\

\indent Let $G$ be a graph, possibly with loops, and let $a\in V(G)$.  We have two ways of counting the degree of this vertex.  We define \emph{deg}$(a):= \left|\{a_1 \in V(G) \mid a_1 \neq a, a_1a \in E(G)\}\right|$, i.e. the number of distinct vertices adjacent to $a$.  Suppose a vertex $a$ has $n$ loops.  We define \emph{degl}$(a):=n+deg(a)$, the sum of the degree and the number of loops.  Given $a,b \in V(G)$, we define $d(a,b)$ to be the shortest path between $a$ and $b$.  If no such path exists, i.e. $a$ and $b$ are in disconnected components of $G$, or the shortest path is infinite, then we say $d(a,b)= \infty$.  We define Diam$(G):=$sup$(\{d(a,b) \mid a,b \in V(G) \}).$  

\subsection{$\tau$-Factorization Definitions}\ 
\\
\indent Let $D$ be a domain.   Let $D^*=D-\{0\}$, let $U(D)$ be the set of units of $D$, and let $D^{\#}=D^*-U(D)$ be the non-zero, non-units of $D$.  We will say $a$ and $b$ are associates if $a=\lambda b$ for some $\lambda \in U(D)$, and denote this $a \sim b$.  Let $\tau$ be a relation on $D^{\#}$, that is, $\tau \subseteq D^{\#} \times D^{\#}$.  We will always assume further that $\tau$ is symmetric.  Let $a$ be a non-unit.  A factorization of the form $a=\lambda a_1 \cdots a_n$ is said to be a \emph{$\tau$-factorization} if $a_i \in D^{\#}$, $\lambda \in U(D)$ and $a_i \tau a_j$ for all $i\neq j$.  If $n=1$, then this is said to be a \emph{trivial $\tau$-factorization}.  Each $a_i$ is said to be a \emph{$\tau$-factor}, or that $a_i$ \emph{$\tau$-divides} $a$, written $a_i \mid_\tau a$.
\\
\indent We say that $\tau$ is \emph{multiplicative} (resp. \emph{divisive}) if for $a,b,c \in D^{\#}$ (resp. $a,b,b' \in D^{\#}$), $a\tau b$ and $a\tau c$ imply $a\tau bc$ (resp. $a\tau b$ and $b'\mid b$ imply $a \tau b'$).  We say $\tau$ is \emph{associate preserving} if for $a,b,b'\in D^{\#}$ with $b\sim b'$, $a\tau b$ implies $a\tau b'$.  We define a \emph{$\tau$-refinement} of a $\tau$-factorization $\lambda a_1 \cdots a_n$ to be a factorization of the form 
$$(\lambda \lambda_1 \cdots \lambda_n) \cdot b_{11}\cdots b_{1m_1}\cdot b_{21}\cdots b_{2m_2} \cdots b_{n1} \cdots b_{nm_n}$$
where $a_i=\lambda_i b_{i1}\cdots b_{im_i}$ is a $\tau$-factorization for each $i$.  This is slightly different from the original definition in \cite{Frazier} where no unit factor was allowed.  One can see they are equivalent when $\tau$ is associate preserving.  We then say that $\tau$ is \emph{refinable} if every $\tau$-refinement of a $\tau$-factorization is a $\tau$-factorization.  We say $\tau$ is \emph{combinable} if whenever $\lambda a_1 \cdots a_n$ is a $\tau$-factorization, then so is each $\lambda a_1 \cdots a_{i-1}(a_ia_{i+1})a_{i+2}\cdots a_n$.  
\\
\indent We pause briefly to give some examples of particular relations $\tau$.

\begin{example}Let $D$ be a domain and let $\tau=D^{\#}\times D^{\#}$. \end{example}
\indent This yields the usual factorizations in $D$ and $\mid_{\tau}$ is the same as the usual divides. Moreover, $\tau$ is multiplicative and divisive (hence associate preserving).

\begin{example}Let $D$ be a domain and let $\tau=\emptyset$. \end{example}
\indent For every $a\in D^{\#}$, there is only the trivial factorization.  Furthermore, all $\tau$-divisors of $a$ are associate to $a$.  If $b\mid{_\tau} a$, then the factorization is forced to be of the form $a=\lambda b=\lambda(\lambda^{-1}(a))$ for $\lambda \in U(D)$. That is, $a\sim b$.  Again $\tau$ is both multiplicative and divisive (vacuously).

\begin{example}Let $D$ be a domain and let $S$ be a non-empty subset of $D^{\#}$.  Let $\tau=S\times S$.  Define $a\tau b \Leftrightarrow a,b\in S$. \end{example}
In this case, $\tau$ is multiplicative (resp. divisive) if and only if $S$ is multiplicatively closed (resp. closed under non-unit factors).  A non-trivial $\tau$-factorization is (up to unit factors) a factorization into elements from $S$.  Some examples of nice sets $S$ might be the set of primes or irreducibles, then a $\tau$-factorization is a prime decomposition or an atomic factorization respectively.

\begin{example} Let $D$ be a domain and let $a \tau b$ if and only if $(a,b)=D$\end{example}
\indent In this case we get the comaximal factorizations studied by S. McAdam and R. Swan in \cite{Mcadam}.  More generally, as in J. Juett in \cite{Juettcomax}, we could let $\star$ be a star-operation on $D$ and define $a\tau b \Leftrightarrow (a,b)^{\star}=D$, that is $a$ and $b$ are $\star$-coprime or $\star$-comaximal. 

\indent Let $a\in D^{\#}$.  As in \cite{Frazier}, we will say $a$ is \emph{$\tau$-irreducible} or a \emph{$\tau$-atom} if factorizations of the form $a=\lambda (\lambda^{-1} a)$ are the only $\tau$-factorizations of $a$.  Then $D$ is said to be \emph{$\tau$-atomic} if every $a\in D^{\#}$ has a $\tau$-factorization $a=\lambda a_1\cdots a_n$ with $a_i$ being $\tau$-atomic for all $1\leq i \leq n$.  We will call such a factorization a \emph{$\tau$-atomic-factorization}.  We say $D$ satisfies \emph{$\tau$-ascending chain condition on principal ideals ($\tau$-ACCP)} if for every chain $(a_0) \subseteq (a_1) \subseteq \cdots \subseteq (a_i) \subseteq \cdots$ with $a_{i+1} \mid_{\tau} a_i$, there exists an $N\in \N$ such that $(a_i)=(a_N)$ for all $i>N$.
\\
\indent A domain $D$ is said to be a \emph{$\tau$-unique factorization domain ($\tau$-UFD)} if (1) $D$ is $\tau$-atomic and (2) for every $a \in D^{\#}$ any two $\tau$-atomic factorizations $a=\lambda_1 a_1 \cdots a_n = \lambda_2 b_1 \cdots b_m$ have $m=n$ and there is a rearrangement so that $a_i$ and $b_i$ are associate for each $1 \leq i \leq n$.  A domain $D$ is said to be a \emph{$\tau$-half factorization domain ($\tau$-HFD)} if (1) $D$ is $\tau$-atomic and (2) for every $a \in D^{\#}$ any two $\tau$-atomic-factorizations have the same length.  A domain $D$ is said to be a \emph{$\tau$-finite factorization domain ($\tau$-FFD)} if for every $a \in D^{\#}$ there are only a finite number of $\tau$-factorizations up to rearrangement and associate.  A domain $D$ is said to be a \emph{$\tau$-weak finite factorization domain ($\tau$-WFFD)} if for every $a \in D^{\#}$, there are only finitely many $b\in D$ such that $b$ is a $\tau$-divisor of $a$ up to associate.  A domain $D$ is said to be a \emph{$\tau$-irreducible divisor finite domain (idf-domain)} if for every $a \in D^{\#}$, there are only finitely many $\tau$-atomic $\tau$-divisors of $a$ up to associate.  A domain $D$ is said to be a \emph{$\tau$-bounded factorization domain ($\tau$-BFD)} if for every $a \in D^{\#}$, there exists a natural number $N(a)$ such that for any $\tau$-factorization $a=\lambda a_1 \cdots a_n$, $n \leq N(a)$.
\\
\indent We have the following set of relationships between the above finite factorization properties from \cite{Frazier}, where $\nabla$ indicates the relationship requires $\tau$ to be refinable and associate preserving.
$$\xymatrix{
            &        \tau \text{-HFD} \ar@{=>}^{\nabla}[dr]     &             &                  &                 \\ 
\tau\text{-UFD} \ar@{=>}[ur] \ar@{=>}^{\nabla}[dr]  &  & \tau\text{-BFD} \ar@{=>}[r]^{\nabla}& \tau\text{-ACCP} \ar@{=>}^{\nabla}[r]& \tau\text{-atomic}\\
 & \tau\text{-FFD} \ar@{=>}[ur]\ar@{=>}[r] &\tau\text{-WFFD}\ar@{=>}[dl] & &  \\
 &\tau\text{-atomic }\tau\text{-idf domain}\ar@{=>}^{\nabla}[u]& & & }$$
 
\section{The $\tau$-Irreducible Divisor Graph and Examples}
Let $D$ be a domain and let $\tau$ be a symmetric and associate preserving relation on $D^{\#}$.  Let \emph{$Irr_\tau(D)$} be the collection of $\tau$-irreducible elements.  We will let $\overline{Irr}_\tau (D)$ be fixed, pre-chosen coset representatives of the cosets $\{a U(D) \mid a \in Irr_\tau(D)\}$.  Given an element $a\in D^{\#}$ with a $\tau$-atomic factorization, we define the \emph{$\tau$-irreducible divisor graph} of $x$ to be $G_\tau (x)=(V,E)$ with $V=\{a\in \overline{Irr}_\tau(D) \mid a \mid_{\tau} x\}$, and given $a_1, a_2 \in \overline{Irr}_{\tau}(D)$, $a_1a_2 \in E$ if and only if there is a $\tau$-factorization of the form $x=\lambda a_1 a_2 \cdots a_n$.  Furthermore, $n-1$ loops will be attached to the vertex corresponding to $a$ if there is a $\tau$-atomic factorization of the form $x=\lambda a \cdots a a_1 \cdots a_n$ where $a$ occurs $n$ times.  Again, if this occurs for arbitrarily large powers of $a$, we allow the possibility of an infinite number of loops.  We then define the \emph{reduced $\tau$-irreducible divisor graph} of $x$ to be the subgraph of $G_\tau(x)$ formed by removing all loops, and denote it $\overline{G_\tau(x)}$.  If there are no $\tau$-atoms in $D$, then $D$ is said to be a \emph{$\tau$-antimatter domain}.  In this case, we define $G_\tau (x)=\overline{G_\tau (x)}=\emptyset$ for all $x \in D^{\#}$.
\\
\begin{example}\label{ex: atom} Let $D$ be a domain.  Suppose $\tau = \emptyset$.\end{example}
By letting $\tau=\emptyset$, we have eliminated all non-trivial $\tau$-factorizations.  This has the effect of making every non-zero, non-unit a $\tau$-atom.  This means for every $x \in D^{\#}$, $G_\tau(x)=(\{x\}, \emptyset).$  The $\tau$-irreducible divisor graph of $x$ consists of the single vertex, $x$ itself (or whichever associate of $x$ chosen initially in $\overline{Irr_\tau(D)}$).

\begin{example}Let $D$ be a domain.  Suppose $\tau = D^{\#} \times D^{\#}$. \end{example}
In this case, we are looking at the usual factorizations in $D$.  This means $a\mid b$ if and only if $a \mid_\tau b$.  Every $\tau$-factorization is a usual factorization and conversely.  Moreover, $x\in D^{\#}$ is $\tau$-atomic if and only if $x$ is atomic.  This results in the irreducible divisor graphs being identical.  Hence, we have $G_\tau(x)=G(x)$ and  $\overline{G_\tau(x)}= \overline{G}(x)$ as defined in Coykendall and Maney \cite{Coykendall}.

\begin{example} (Inspired by \cite[Example 2.4]{Coykendall}) Let $D=\Q[x^2, x^3]$ and let $f(x)=x^8-x^9$ and consider the relation $\tau$ defined by $g(x) \tau g'(x)$ if and only if $deg(g(x))=deg(g'(x))$. \end{example}
The irreducible factorizations of $f(x)$ are
$$f(x)=x^8-x^9=x^2\cdot x^2 \cdot x^2 \cdot (x^2-x^3) = (x^3-x^4)\cdot x^2 \cdot x^3=(x^2 - x^3) \cdot x^3 \cdot x^3.$$
It is clear that only the last factorization above is a $\tau$-factorization.  It is the only factorization in which all of the factors have the same degree.  Every irreducible is certainly $\tau$-irreducible, but the question is: are there any additional $\tau$-atomic elements which $\tau$-divide $f(x)$?  For instance, $g(x)=x^5$ is $\tau$-atomic since the only non-trivial factorization up to associates and rearrangement is $g(x)=x^2 \cdot x^3$ is thrown out due to the factors having different degrees.  Any $\tau$-factorization of $f(x)$ must have $\tau$-factors of the same degree.  This amounts to finding a proper partition of $deg(f(x))=9$ into parts of equal size.  This is done only by nine degree $1$ parts, or three degree $3$ parts. Since all non-units in $D$ have degree at least $2$, the former is not possible in this ring.  This leaves only one possible $\tau$-atomic factorization up to associates and rearrangement:
$$f(x)=(x^2 - x^3) \cdot x^3 \cdot x^3.$$
\indent We show both $G(f(x))$ and $G_\tau(f(x))$ below in Figure 1 for the sake of comparison.

 \begin{figure}[H]
	 \centering
		 \includegraphics[scale=.7]{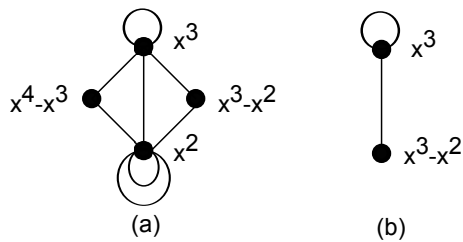}
	 \caption{(a) $G(f(x))$, (b) $G_\tau (f(x))$}
	 \label{fig:Irreducible divisor graphs}
 \end{figure}

\indent This yields the following corresponding reduced irreducible divisor graphs $\overline{G}(f(x))$ and $\overline{G_\tau (f(x))}$ in Figure 2.

 \begin{figure}[H]
	 \centering
		 \includegraphics[scale=.7]{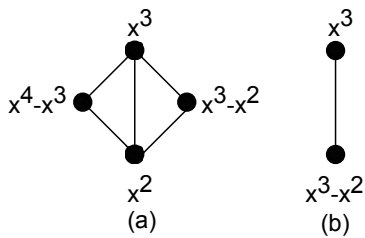}
	 \caption{(a) $\overline{G}(f(x))$, (b) $\overline{G_\tau (f(x))}$}
	 \label{fig:Reduced Irreducible divisor graphs}
 \end{figure}

\begin{remark}
This above partition argument shows that any polynomial of prime degree $p$ is a $\tau$-atom.  The only way to non-trivially partition a prime integer $p$ into equal parts is $p$ parts of degree $1$ each.  But all non-unit elements have degree at least $2$.  Hence there can only be the trivial factorization.  Thus we know the $\tau$-irreducible divisor graph of any polynomial of prime degree, $p(x)$ consists of the single vertex, $p(x)$ (or possibly whichever associate of $p(x)$ was taken in $\overline{Irr_\tau(D)}$).  It is perhaps worth noting that the converse is false as conveniently demonstrated above by $x^4-x^3$ being irreducible and hence $\tau$-irreducible, despite having composite degree.
\end{remark}
\section{The $\tau$-Irreducible Divisor Graph and $\tau$-Finite Factorization Properties}
We first observe that the $\tau$-irreducible divisor graph yields a new equivalent characterization of a $\tau$-irreducible element of \cite{Frazier}.  This formalizes something we observed in Example \ref{ex: atom}.
\begin{theorem}\label{thm: atom} Let $D$ be a domain and $\tau$ a symmetric, associate preserving relation on $D^{\#}$.  If $D$ is $\tau$-atomic, then a non-unit $x$ is $\tau$-irreducible if and only if $G_\tau(x) \cong K_1$, the complete graph on a single vertex which is some associate of $x$.
\end{theorem}
\begin{proof} If $x$ is $\tau$-irreducible, then the only $\tau$-factorizations are of the form $x=\lambda (\lambda^{-1}x)$, hence the only $\tau$-divisors of $x$ are associates of $x$.  Moreover, $x$ is a non-zero, non-unit, so $x^2$ does not divide $x$, hence there are no loops.  Thus, $G_\tau (x)$ is a single isolated vertex generated by whichever associate of $x$ was chosen in $\overline{Irr_\tau(D)}$.  Conversely, suppose $x \in R$ is a non-unit such that $G_\tau(x) \cong K_1$.  We suppose $x$ were not $\tau$-irreducible.  Then there is a non-trivial $\tau$-atomic factorization of the form $x=\lambda a_1 \cdots a_n$ with $n \geq 2$.  This yields $a_1, a_2 \in V(G_\tau(x))$, but there is only one vertex and no loops in $G_\tau (x)$, contradicting the hypothesis and completing the proof.
\end{proof}
\begin{theorem}\label{thm: accp} Let $D$ be a domain and $\tau$ a symmetric, refinable and associate preserving relation on $D^{\#}$.  If $D$ is $\tau$-atomic such that for all $x \in D^{\#}$, and for all $a \in V(G_\tau(x))$, degl$(a) < \infty$, then $D$ satisfies $\tau$-ACCP.
\end{theorem}
\begin{proof} Suppose $D$ did not satisfy $\tau$-ACCP.  Then there exists a chain of principal ideals $(x_1) \subsetneq (x_2) \subsetneq (x_3) \subsetneq \cdots$ such that $x_{i+1} \mid_\tau x_i$.  Say 

\begin{equation}\label{1} x_i = \lambda_i x_{i+1}\cdot a_{i1} \cdots a_{in_i}\end{equation}

is a $\tau$-factorization for each $i$.  Because $D$ is $\tau$-atomic and $\tau$ is refinable and associate preserving, we may replace each $a_{ij}$ with a $\tau$-atomic factorization.  This allows us to assume each factor in Equation \eqref{1} is a $\tau$-atom.  Since $\tau$ is associate preserving, we may assume further that each $a_{ij} \in \overline{Irr}_\tau(D)$.  Because $\tau$ is refinable, 

\begin{equation}\label{2}x_1 = \lambda_1 x_{2}\cdot a_{11} \cdots a_{1n_1}=(\lambda_1 \lambda_2) x_{3}\cdot a_{21} \cdots a_{2n_2}\cdot a_{11} \cdots a_{1n_1}= \cdots\end{equation}

are all $\tau$-factorizations with $a_{ij}$ $\tau$-atomic.  Because $x_i \subsetneq x_{i+1}$, we must have in Equation \eqref{1} that $n_i \geq 1$ or else $x_i \sim x_{i+1}$.  This means the factorizations in each iteration of Equation \eqref{2} increase in length.  If $\{a_{ij}\}$ is infinite, then $a_{11}$ has an infinite number of adjacent vertices in $G_\tau(x_1)$, i.e $degl(a_{11}) \geq deg(a_{11})=\infty$.  Otherwise, if $\{a_{ij}\}$ is finite, then one of the $a_{i_0j_0}$ for some $i_0$ and $j_0$ occurs an infinite number of times.  Hence degl$(a_{i_0j_0})=\infty$ in $G_\tau(x_1)$ since arbitrarily high powers of $a_{i_0j_0}$ $\tau$-divide $x_1$.  This is a contradiction and $D$ must satisfy $\tau$-ACCP as desired.
\end{proof}
\begin{theorem}\label{thm: ufd}Let $D$ be a domain and $\tau$ a symmetric and associate preserving relation on $D^{\#}$.  If $D$ is $\tau$-atomic, then the following are equivalent.
\\
(1) $D$ is a $\tau$-UFD.
\\
(2) $G_\tau(x)$ is a pseudoclique for every $x\in D^{\#}$.
\\
(3) $\overline{G_\tau(x)}$ is a clique for every $x\in D^{\#}$.
\\
(4) $G_\tau (x)$ is connected and Diam$(G_\tau (x))=1$ for every $x\in D^{\#}$.
\\
(5) $\overline{G_\tau(x)}$ is connected and Diam$(\overline{G_\tau(x)})=1$ for every $x\in D^{\#}$.
\\
(6) $G_\tau(x)$ is connected for every $x\in D^{\#}$.
\\
(7) $\overline{G_\tau(x)}$ is connected for every $x\in D^{\#}$.
\end{theorem}
\begin{proof} (1) $\Rightarrow$ (2) Let $x\in D^{\#}$.  Let $x=\lambda a_1 \cdots a_m$ be the unique $\tau$-atomic factorization of $x$ up to rearrangement and associates.  Because $\tau$ is associate preserving, one may adjust the unit factor in front if necessary, to assume without loss of generality $a_i\in \overline{Irr}(D)$ for each $i$.  After rearrangement, we may take the first $n$ $\tau$-factors to be distinct up to associate, and assume the last $m-n$ $\tau$-factors are repeated associates of $a_1, \ldots, a_n$.  Since these are the only irreducible $\tau$-divisors of $x$, $V(G_\tau(x))=\{a_1, \ldots, a_n\}$.  It is clear that given any two vertices, $a_i, a_j \in V(G_\tau(x))$, $a_ia_j \in  E(G_\tau(x))$, because $a_i \tau a_j$ and $a_i a_j \mid_\tau x$.  If there are repeated associates in the $\tau$-factorization, say $x= \lambda a_1^{e_1} \cdots a_n^{e_n}$ with $n \leq m$ is a factorization with $e_1 + e_2 + \ldots + e_n=m$.  Then if $e_i >1$, we have $a_i \tau a_i$ and $a_ia_i \mid_{\tau} x$, so we have $e_i - 1$ loops attached to $a_i$.  Hence $G_\tau (x)$ is a pseudoclique.
\\
\indent It is immediate from the definition of the reduced $\tau$-irreducible divisor graph that (2) $\Leftrightarrow$ (3) and $(6) \Leftrightarrow (7)$. Furthermore, $(2)\Leftrightarrow (4)$ and $(3)\Leftrightarrow (5)$ respectively since a connected graph is complete if and only if the diameter is $1$.
\\
\indent If $G_\tau(x)$ is complete it is certainly connected, so (2) $\Rightarrow$ (6).  It now suffices to show that (6) $\Rightarrow$ (1).  The following is a modification of the proof of \cite[Theorem 2.1]{Axtellidgd}.  Let $\mathcal{A}$ be the set of all non-zero, non-units which admit at least two distinct $\tau$-atomic factorizations.  We show $\mathcal{A}=\emptyset$.  Suppose otherwise and let $n:=min_{x\in \mathcal{A}}\{k \mid x=\lambda a_1 \cdots a_k \text{ is a } \tau\text{-atomic factorization}\}$.  Clearly, $n\geq 2$.  Let $y\in \mathcal{A}$ such that $y=\lambda a_1 \cdots a_n$.  Then there is a distinct $\tau$-atomic factorization $y=\mu b_1 \cdots b_t$ with $t \geq n$.  There is a path in $G_\tau(x)$ connecting $b_1$ and $a_1$, so without loss of generality we may as well have picked at $\tau$-atomic factorization such that $b_1$ and $a_1$ are actually adjacent in $G_\tau(x)$.  If $a_i$ and $b_j$ were associates for any $1 \leq i \leq n$, $1 \leq j \leq t$, then $\frac{y}{a_i}=\lambda'a_1 \cdots \widehat{a_i} \cdots a_n = \mu' b_1 \cdots \widehat{b_k} \cdots b_t$ provides two distinct $\tau$-atomic factorizations of a non-zero, non-unit but contradicts the minimality of $n$.  Since we chose $a_1$ and $b_1$ to be adjacent, we know there is a $\tau$-atomic factorization of the form $y=\gamma a_1b_1 c_1 \cdots c_m=\lambda a_1 \cdots a_n$.  But this again yields two distinct $\tau$-atomic factorizations of $\frac{y}{a_1}$:
$$ \frac{y}{a_1}=\gamma b_1 c_1 \cdots c_m=\lambda a_2 \cdots a_n$$
again contradicting the minimality of $n$.  Hence, $\mathcal{A}$ must be empty as desired, completing the proof.
\end{proof}

\begin{theorem}\label{thm: ffd} Let $D$ be a domain and $\tau$ a symmetric, refinable and associate preserving relation on $D^{\#}$.  If $D$ is $\tau$-atomic, then consider the following statements.
\\
(1) $G_\tau(x)$ is finite for every $x \in D^{\#}$.
\\
(2) $\overline{G_\tau(x)}$ is finite for every $x \in D^{\#}$.
\\
(3) $D$ is a $\tau$-irreducible divisor finite domain.
\\
(4) $D$ is a $\tau$-weak finite factorization domain. 
\\
(5) $D$ is a $\tau$-finite factorization domain.
\\
(6) For all $x\in D^{\#}$, degl$(a)< \infty$ for all $a \in V(G_\tau(x))$.
\\
(7) For all $x\in D^{\#}$, deg$(a)< \infty$ for all $a \in V(G_\tau(x))$.
\\
We have (1)-(5) are equivalent, (3) $\Rightarrow$ (6) $\Rightarrow$ (7), and if we assume further that $\tau$ is reflexive, then (7) $\Rightarrow$ (3) and all are equivalent.  
\end{theorem}
\begin{proof} We begin by showing (1)-(5) are equivalent.  $(1) \Leftrightarrow (2)$ and (3) $\Leftrightarrow$ (1) are immediate from the definitions.  (5) $\Rightarrow$ (4) is clear since every every $\tau$-divisor up to associate must appear as a $\tau$-factor in one of the finitely many $\tau$-factorizations.  (4) $\Rightarrow$ (3) every $\tau$-atomic divisor up to associate is certainly among the $\tau$-divisors. 
\\
\indent We need only prove (3) $\Rightarrow$ (5).  We modify the proof of \cite[Theorem 5.1]{anderson90}.  Let $a\in D^{\#}$.  Let $a_1, a_2, \ldots a_n$ be the collection of all non-associate $\tau$-irreducible divisors of $a$.  Because $D$ is $\tau$-atomic and $\tau$ is refinable and associate preserving, given any $\tau$-factorization, we could $\tau$-refine it into a $\tau$-atomic factorization.  In this way, every $\tau$-factorization corresponds to a $\tau$-factorization of the form 
$$a=\lambda a_1^{s_1} \cdots a_n^{s_n}$$
with $0 \leq s_i$.  If we can show that the set of factorizations of this form is finite, then every $\tau$-factorization occurs as some grouping of these $\tau$-atomic factors.  
\\
\indent We suppose first that for each $i$, there is a bound $N_i$ such that every $\tau$-refinement of a $\tau$-factorization as above has $0 \leq s_i \leq N_i$.  Then this set is finite, with $N_1 \cdot N_2 \cdots N_n$ elements in it.  So now we must have some $s_i$ which is unbounded, say $s_1$ is unbounded.  Then for each $k \geq 1$, we can write
$$x=\lambda_k a_1^{s_{k1}} \cdots a_n^{s_{kn}}$$
$\tau$-factorizations with $s_{k1} < s_{k2} < s_{k3} < \cdots$.  Suppose that in the set of factorizations, $\{ s_{ki} \}$ is bounded for $1 < i \leq n$.  Then there are only finitely many choices for $s_{k2}, \ldots, s_{kn}$, we must have $s_{k2}=s_{j2}, s_{k3}=s_{j3}, \ldots, s_{kn}=s_{jn}$ for some $j >k$.  But then we have 
$$x=\lambda_j a_1^{s_{j1}} \cdots a_n^{s_{jn}}=\lambda_k a_1^{s_{k1}} \cdots a_n^{s_{kn}}.$$
Since $D$ is a domain, we can cancel leaving $\lambda_j a_1^{s_{j1}}=\lambda_k a_1^{s_{k1}}$ with $s_{j1} > s_{k1}$, a contradiction.  
\\
\indent This means some set $\{ s_{ki} \}$ is unbounded for a fixed $i$ with $1 < i \leq n$.  Without loss of generality suppose it is $i=2$.  We continue in this manner until we get subsequences with $s_{11} < s_{21} < s_{31} < \cdots$ and $s_{21} < s_{22} < s_{32} < \cdots.$  But this means we have 
$$x=\lambda_1 a_1^{s_{11}} \cdots a_n^{s_{1n}} = \lambda_2 a_1^{s_{21}} \cdots a_n^{s_{2n}}$$
with $s_{1i} < s_{2i}$ for each $i$, a contradiction.  Thus this set must be finite, and the proof is complete.
\\
\indent (5) $\Rightarrow$ (6) $\Rightarrow$ (7) is immediate. 
\\
\indent We now suppose in addition that $\tau$ is reflexive and show (7) $\Rightarrow$ (3).  Suppose there is a $a\in D^{\#}$ such that $\{a_i \}_{i\in I}$ is an infinite collection of non-associate $\tau$-atomic divisors of $a$.  Suppose for each $i \in I$, that $a= \lambda_i a_i a_{i1} \cdots a_{in_i}$ are the given $\tau$-factorizations showing $a_i$ is a $\tau$-atomic divisor of $a$.  But then since $\tau$ is reflexive, we have $a \tau a$, so
$$a^2=a\cdot a = \left(\lambda_i a_i a_{i1} \cdots a_{in_i}\right) \left(\lambda_j a_j a_{j1} \cdots a_{jn_j}\right)$$ is a $\tau$-factorization for any choice of $i,j \in I$.  Because $\tau$-is refinable and associate preserving, we have
$$a^2= (\lambda_i\lambda_j) a_i a_j a_{i1} \cdots a_{in_i} a_{j1} \cdots a_{jn_j}$$
is a $\tau$-factorization, showing $a_i$ and $a_j$ are adjacent in $G_\tau (a^2)$, for any choice of $i,j\in I$.  Thus when we fix $i=i_0$, and let $j$ range over all possible choices of $a_j \in \{a_i \}_{i\in I}$, we see $a_{i_0}a_j \in E(G_\tau (a^2))$ for all $j \in I$, $j \neq i_0$.  Hence deg$(a_{i_0})=\infty$ in $G_\tau(a^2)$. 
\end{proof}

\begin{corollary} Let $D$ be a domain and $\tau$ a symmetric, refinable and associate preserving relation on $D^{\#}$.  If $D$ is $\tau$-atomic and if $G_\tau (x)$ (resp. $\overline{G_\tau(x)}$) is connected for every $x \in D^{\#}$, then $G_\tau (x)$ (resp. $\overline{G_\tau(x)}$) is a finite pseudoclique for every $x \in D^{\#}$.  Furthermore, $\overline{G_\tau(x)}=K_n$ for some finite $n$.
\end{corollary}
\begin{proof} If $G_\tau (x)$ is connected for every $x\in D^{\#}$, then by Theorem \ref{thm: ufd} $D$ is a $\tau$-UFD which implies $D$ is a $\tau$-atomic $\tau$-irreducible divisor finite ring.  Theorem \ref{thm: ffd} shows $G_\tau(x)$ has a finite number of vertices.  Moreover, by Theorem \ref{thm: ufd}, each of these vertices is adjacent to every other vertex, i.e. $G_\tau(x)$ is a finite a pseudoclique.  The last statement is immediate.
\end{proof}
The following result is well known and was proven in \cite{Frazier}, but we include it since the $\tau$-irreducible divisor graph results yield a nice proof. 
\begin{corollary}Let $D$ be a domain and $\tau$ a symmetric, refinable and associate preserving relation on $D^{\#}$.  If $D$ is a $\tau$-atomic, $\tau$-irreducible divisor finite domain, then $D$ satisfies $\tau$-ACCP.
\end{corollary}
\begin{proof}Theorem \ref{thm: ffd} shows for an associate preserving and refinable $\tau$, a $\tau$-atomic $\tau$-idf domain has the property that for all $x\in D^{\#}$, $degl(a)<\infty$ for any vertex $a\in V(G_\tau(x)).$  By Theorem \ref{thm: accp}, this shows $D$ satisfies $\tau$-ACCP.
\end{proof}
\section*{Acknowledgment}
I would like to thank the referee for their helpful comments and careful reading which have improved the quality of this article.  This research was conducted as a research fellow under the supervision of Professor Daniel D. Anderson while at The University of Iowa.


\begin{thebibliography}{9}
\bibitem{anderson90} D.D. Anderson, D.F. Anderson, and M. Zafrullah, \emph{Factorization in Integral Domains}, J. Pure Appl. Algebra., \textbf{69} (1990), 1-19.

\bibitem{Frazier} D. D. Anderson and Andrea M. Frazier, \emph{On a general theory of factorization in integral domains}, Rocky Mountain J. Math., \textbf{41} (2011), 663--705.
 
\bibitem{andersonzdg} D.D. Anderson and M. Naseer, \emph{Beck's coloring of a commutative ring}, J. Algebra, \textbf{159} (1993), 500--514.
    
\bibitem{davidanderson} D.F. Anderson, A. Frazier, A. Lauve, and P.S. Livingston, \emph{The zero-divisor graph of a commutative ring. {II}}, Ideal theoretic methods in commutative algebra ({C}olumbia,{MO}, 1999), Lecture Notes in Pure and Appl. Math., \textbf{220}, Dekker, New York, (2001), 61--72.

\bibitem{Livingston} D.F. Anderson and P.S. Livingston, \emph{The zero-divisor graph of a commutative ring}, J. Algebra, \textbf{217} (1999), 434--447.

\bibitem{Axtellidgd} M. Axtell, N. Baeth, and J. Stickles, \emph{Irreducible divisor graphs and factorization properties of domains}, Comm. Algebra, \textbf{39} (2011), 4148--4162.

\bibitem{Axtellidgzd} M. Axtell and J. Stickles, \emph{Irreducible divisor graphs in commutative rings with zero-divisors}, Comm. Algebra, \textbf{36} (2008), 1883--1893.

\bibitem{Beck} I. Beck, \emph{Coloring of commutative rings}, J. Algebra, \textbf{116} (1988), 208--226.

\bibitem{Coykendall} J. Coykendall and J. Maney, \emph{Irreducible divisor graphs}, Comm. Algebra, \textbf{35} (2007), 885--895.

\bibitem{Juettcomax} J. Juett, \emph{Generalized comaximal factorization of ideals}, J. Algebra, \textbf{352} (2012), 141--166.

\bibitem{Mcadam} S. McAdam and R. Swan, \emph{Unique comaximal factorization}, J. Algebra, \textbf{276} (2004), 180--192.

\bibitem{Mooney} C.P. Mooney, \emph{Generalized factorization in commutative rings with zero-divisors}, Houston J. Math., to appear. 

\bibitem{Mooney2} C.P. Mooney, \emph{Generalized u-factorization in commutative rings with zero-divisors}, Rocky Mountain J. Math., to appear. 
\end{thebibliography}
\end{document}